\documentclass[final,1p,times,authoryear]{elsarticle}
\usepackage[utf8]{inputenc}
\usepackage[T1]{fontenc}
\usepackage[shortlabels]{enumitem}
\usepackage{amsmath}
\usepackage{amsthm}
\usepackage{amssymb}
\usepackage{amsfonts}
\usepackage{leftidx}
\usepackage{mathtools}
\usepackage{tikz-cd}
\usepackage[colorlinks=true, linkcolor=red, citecolor=blue, urlcolor=cyan]{hyperref}
\usepackage[ruled,lined,linesnumbered]{algorithm2e}

\usepackage[nameinlink]{cleveref} % better referencing

\newtheorem{theorem}{Theorem}
\newtheorem{lemma}[theorem]{Lemma}
\newtheorem{corollary}[theorem]{Corollary}

\newtheorem{proposition}[theorem]{Proposition}
\newtheorem{definition}[theorem]{Definition}

\newtheorem{example}[theorem]{Example}
\newtheorem{remark}[theorem]{Remark}

\DeclareMathOperator{\im}{im}
\DeclareMathOperator{\LSat}{LSat}
\DeclareMathOperator{\lc}{lc}
\DeclareMathOperator{\lm}{lm}
\DeclareMathOperator{\lexp}{le}
\DeclareMathOperator{\NF}{NF}
\DeclareMathOperator{\LeftNF}{LeftNF}
\DeclareMathOperator{\Ann}{Ann}
\DeclareMathOperator{\Mon}{Mon}

\DeclarePairedDelimiter{\erz}{\langle}{\rangle}
\DeclarePairedDelimiter{\merz}{[}{]}
\DeclarePairedDelimiter{\set}{\{}{\}}

\newcommand{\IN}{\mathbb{N}}
\newcommand{\tensor}{\otimes}
\newcommand{\inv}{^{-1}}

\newcommand{\nset}[2][1]{\set{#1,\ldots,#2}}
\newcommand{\leftideal}[2]{\leftidx{_{#1}}{\erz{#2}}{}}
\newcommand{\ideal}[1]{\erz{#1}}
\newcommand{\central}[1]{Z(#1)}
\newcommand{\mf}[1]{\mathfrak{#1}}

\newcommand{\compset}[1]{\underline{#1}}
\newcommand{\ModExp}[2]{\compset{#1}\times\IN_0^{#2}}
\newcommand{\POTify}[1]{\operatorname{{#1}^{\operatorname{POT}}}}
\newcommand{\myemph}[1]{\textit{\textbf{#1}}}
\renewcommand{\emph}{\myemph}

\begin{document}

\begin{frontmatter}

\title{Constructive Arithmetics in Ore Localizations Enjoying Enough Commutativity}

\author{Johannes Hoffmann}
\address{Research group in Free Probability, Saarland University}
\ead{Johannes.Hoffmann@math.uni-sb.de}
\ead[url]{https://www.math.uni-sb.de/ag/speicher/hoffmannE.html}

\author{Viktor Levandovskyy}
\address{Lehrstuhl D f\"ur Mathematik, RWTH Aachen University}
\ead{Viktor.Levandovskyy@math.rwth-aachen.de}
\ead[url]{http://www.math.rwth-aachen.de/~Viktor.Levandovskyy/}

\begin{abstract}
	This paper continues a research program on constructive investigations of non-commutative Ore localizations, initiated in our previous papers, and particularly touches the constructiveness of arithmetics within such localizations.
	Earlier we have introduced monoidal, geometric and rational types of localizations of domains as objects of our studies. 
	Here we extend this classification to rings with zero divisors and consider Ore sets of the mentioned types which are commutative enough: such a set either belongs to a commutative algebra or it is central or its elements commute pairwise.
	By using the systematic approach we have developed before, we prove that arithmetic within the localization of a commutative polynomial algebra is constructive and give the necessary algorithms.
	We also address the important question of computing the local closure of ideals which is also known as the desingularization, and present an algorithm for the computation of the symbolic power of a given ideal in a commutative ring.
	We also provide algorithms to compute local closures for certain non-commutative rings with respect to Ore sets with enough commutativity.
\end{abstract}

\end{frontmatter}

\section*{Introduction}

The algebraic technique of commutative localization has found applications across many areas of mathematics and beyond; it is instrumental everywhere from algebraic geometry to system and control theory. Among several possible generalizations to the non-commutative setting, Ore localization stands out as being approachable in a constructive manner by methods of modern computer algebra. This paper is a part of our broad program dedicated to realizing this approach. Starting point was the investigation of arithmetic operations with left and right fractions in Ore localizations of non-commutative domains in \cite{HL_ISSAC17} and its extended version \cite{HL17ext}. We have demonstrated that such arithmetic operations are based essentially on two algorithms, namely 
\begin{itemize}
	\item the computation of the kernel of a module homomorphism and 
	\item the computation of the intersection of a left ideal with a monoid.
\end{itemize}
Especially the latter algorithm is hardly constructive in a broad generality, therefore we have introduced a partial classification of types of multiplicative monoids for which the intersection problem can be solved algorithmically.
We recall an extended version of the classification in \Cref{classification_of_localization_types}.

In this paper we revisit the case of commutative polynomial algebras both on their own and as homomorphic images in a noncommutative ring as either central subalgebras or 
those which are generated by pairwise commuting elements.
On the one hand we extend our framework to such algebras with zero divisors.
On the other hand we also consider the important problem of the computation of the local closure of a submodule with respect to a given denominator set (also known as the desingularization), which is tightly connected with the generalized torsion submodule of a module.

Though some of the algorithms have been known in commutative algebra, they are scattered in the existing literature and are often deprived of proofs.
We describe the problems in a systematic and self-contained way.
In the collection of the algorithms we present, \ref{alg_NCIdealIntersectionWithMonoid}, \ref{alg:decomposition_closure}, \ref{alg:symbolic_power}, \ref{alg:CentralEssentialRationalClosure}, \ref{alg_CentralWeylClosure}, and \ref{alg_AnnFsViaWeylClosure}
are new.
The following list summarizes the problems discussed in this paper with references to the corresponding algorithms:

\begin{description}[leftmargin=0cm]
	\item[Polynomial algebras:]
		In a polynomial algebra $R=K[x]/J$, where $J$ is an ideal in the commutative polynomial ring $K[x]:=K[x_1,\ldots,x_n]$, we can compute the intersection of an ideal $I$ in $R$ with a multiplicative subset $S$ of $R$, if
		\begin{itemize}
			\item
				$S\inv R$ is monoidal and $S$ is finitely generated (\Cref{alg_NCIdealIntersectionWithMonoid}),
			\item
				$S\inv R$ is geometric (\Cref{alg:geometric_intersection}), or
			\item
				$S\inv R$ is essential rational (\Cref{alg:rational_intersection}).
		\end{itemize}
		Furthermore, we can decide whether a multiplicative submonoid of $R$ contains $0$  (\Cref{alg_ZeroContainedInMonoid}). It is important, since localizing $R$ at a submonoid $S$ containing $0$ yields the trivial localization $S\inv R=\set{0}$.
	\item[Commutative rings:]
		In an arbitrary commutative ring $R$ we can compute the closure of an ideal $I$ with respect to a multiplicative set $S$ via \Cref{alg:decomposition_closure} under the following conditions:
		\begin{enumerate}[(1)]
			\item
				The ideal $I$ is decomposable into primary ideals and such a decomposition is either known or computable.
			\item
				We can decide whether $Q\cap S=\emptyset$ for any primary ideal $Q$ in $R$.
		\end{enumerate}
	\item[] In particular, we give \Cref{alg:symbolic_power} for computing the symbolic power of a given ideal.	
	\item[G-algebras:]
		In a $G$-algebra we can compute the closure of an ideal $I$ with a left Ore set $S$, if
		\begin{itemize}
			\item
				$S\inv R$ is monoidal and $S$ is generated as a monoid by finitely many elements $f_1,\ldots,f_k$ that commute pairwise and $\central{A}\cap S$ contains a multiple of $f_1\cdot\ldots\cdot f_k$ (\Cref{monoidal_closure_in_G-algebras}), or
			\item
				$S\inv R$ is central essential rational (\Cref{alg:CentralEssentialRationalClosure}).
		\end{itemize}
\end{description}

In comparison to the ISSAC version of this paper (\cite{HL_ISSAC18}), the material has been expanded and slightly reworked (some proofs now contain more details).
In particular, we have expanded \Cref{sectAppLocalClosureComm} with criteria for emptiness of the intersection of primary ideals and multiplicative sets as well as with \Cref{alg:symbolic_power} (computation of the symbolic power of an ideal), 
described Weyl closure algorithms in \Cref{remWeylClosureAlgorithms}, added \Cref{subsectCentralWeylclosure} on the details of central Weyl closure including \Cref{alg_CentralWeylClosure}, and finally added \Cref{AnnfsViaCentralWeylclosure} with the new \Cref{alg_AnnFsViaWeylClosure} to compute the annihilator ideal of the important special function $f^s$.

\section{The basics of (Ore) localization}

All rings are assumed to be associative and unital, but not necessarily commutative.

\begin{definition}
	A subset $S$ of a ring $R$ is called
	\begin{itemize}
		\item
			a \emph{multiplicative set} if $1\in S$, $0\notin S$ and for all $s,t\in S$ we have $s\cdot t\in S$.
		\item
			a \emph{left Ore set} if it is a multiplicative set that satisfies the \emph{left Ore condition}: for all $s\in S$ and $r\in R$ there exist $\tilde{s}\in S$ and $\tilde{r}\in R$ such that $\tilde{s}r=\tilde{r}s$.
		\item
			a \emph{left denominator set} if it is a left Ore set that is additionally \emph{left reversible}: for all $s\in S$ and $r\in R$ such that $rs=0$ there exists $\tilde{s}\in S$ satisfying $\tilde{s}r=0$.
	\end{itemize}
\end{definition}

For any subset $B$ of $R\setminus\set{0}$ we can consider the set $\merz{B}$ consisting of all finite products of elements of $B$, where the empty product represents $1$.
If $R$ is a domain then $\merz{B}$ is always a multiplicative set called the \emph{multiplicative closure} of $B$.

The main goal of localization can be seen from the following axiomatic definition:

\begin{definition}\label{axiomatic_definition_of_Ore_localization}
	Let $S$ be a multiplicative subset of a ring $R$.
	A ring $R_S$ together with a homomorphism $\varphi:R\rightarrow R_S$ is called a \emph{left Ore localization} of $R$ at $S$ if:
	\begin{enumerate}[(1)]
		\item
			For all $s\in S$, the element $\varphi(s)$ is a unit in $R_S$.
		\item
			For all $x\in R_S$, there exist $s\in S$ and $r\in R$ such that $x=\varphi(s)\inv\varphi(r)$.
		\item
			We have $\ker(\varphi)=\set{r\in R\mid\exists~s\in S:sr=0}$.
	\end{enumerate}
\end{definition}

One can show that a left Ore localization of $R$ at $S$ exists if and only if $S$ is a left denominator set.
In this case the localization is unique up to isomorphism.
The classical construction is given by the following:

\begin{theorem}\label{construction_of_the_Ore_localization}
	Let $S$ be a left denominator set in a ring $R$.
	The relation $\sim$ on $S\times R$, given by
	\[
		(s_1,r_1)\sim(s_2,r_2)\Leftrightarrow\exists~\tilde{s}\in S\exists~\tilde{r}\in R:\tilde{s}s_2=\tilde{r}s_1\text{ and }\tilde{s}r_2=\tilde{r}r_1,
	\]
	is an equivalence relation.
	Now $S\inv R:=((S\times R)/\sim,+,\cdot)$ becomes a ring via
	\[
		(s_1,r_1)+(s_2,r_2)
		:=(\tilde{s}s_1,\tilde{s}r_1+\tilde{r}r_2),
	\]
	where $\tilde{s}\in S$ and $\tilde{r}\in R$ satisfy $\tilde{s}s_1=\tilde{r}s_2$, and
	\[
		(s_1,r_1)\cdot(s_2,r_2)
		:=(\tilde{s}s_1,\tilde{r}r_2),
	\]
	where $\tilde{s}\in S$ and $\tilde{r}\in R$ satisfy $\tilde{s}r_1=\tilde{r}s_2$.
	Together with the \emph{localization map} or \emph{structural homomorphism}
	\[
		\rho_{S,R}:R\rightarrow S\inv R,\quad
		r\mapsto(1,r),
	\]
	the pair $(S\inv R,\rho_{S,R})$ is the left Ore localization of $R$ at $S$.
\end{theorem}

The elements of $S\inv R$ are called \emph{left fractions} and are denoted again as tuples $(s,r)$ which are identified with their equivalence class modulo $\sim$.
The localizations that appear the most in applications are those with denominator sets of the following three types:

\begin{definition}\label{classification_of_localization_types}
	Let $K$ be a field, $R$ a $K$-algebra and $S$ a left denominator set in $R$.
	Then $S$ (and by extension, the localization $S\inv R$) might belong to one of the following non-exclusive types:
	\begin{description}
		\item[\emph{Monoidal}]
			$S$ is generated as a multiplicative monoid by at most countably many elements.
		\item[\emph{Geometric}]
			$S=(K[x]/J)\setminus\mf{p}$ for some prime ideal $\mf{p}$ in the polynomial algebra $K[x]/J\subseteq R$, where $J$ is an ideal in $K[x]:=K[x_1,\ldots,x_n]$.
		\item[\emph{Rational}]
			$S=T\setminus\set{0}$ for some $K$-subalgebra $T$ of $R$.\\
			\textbf{Special case:}
			If $R$ is generated over $K$ by a set of variables $x=\set{x_1,\ldots,x_n}$ and $T$ is generated by a subset of $x$ we call $S$ an \emph{essential} rational left denominator set.
	\end{description}
\end{definition}

\begin{definition}
	Let $S$ be a left denominator set in a ring $R$ and $M$ a left $R$-module.
	Then the \emph{left Ore localization} of $M$ at $S$ is defined as $S\inv M:=S\inv R\tensor_RM$.
\end{definition}

\begin{lemma}[\normalfont e.g. \cite{skoda_2006}, 7.3]\label{elementary_tensors}
	Let $S$ be a left denominator set in a ring $R$ and $M$ a left $R$-module.
	Any element of $S\inv M$ can be written in the form $(s,1)\tensor m$ for some $s\in S$ and $m\in M$.
\end{lemma}

\Cref{elementary_tensors} allows us to write $(s,m)$ for an element in $S\inv M$ in analogy to the notation for elements of $S\inv R$.

Alternatively, one can define localization of modules similar to the axiomatic approach in \Cref{axiomatic_definition_of_Ore_localization}, prove its uniqueness and give an elementary construction like in \Cref{construction_of_the_Ore_localization}.

\begin{definition}
	Let $S$ be a left denominator set in a ring $R$ and $M$ a left $R$-module.
	\begin{itemize}
		\item
			The \emph{localization map} of $M$ with respect to $S$ is the homomorphism of left $R$-modules
			\[
				\varepsilon:=\varepsilon_{S,R,M}:M\rightarrow S\inv M,\quad
				m\mapsto(1,m),
			\]
			with kernel $\set{m\in M\mid\exists~s\in S:sm=0}$.
		\item
			Let $P$ be a left $R$-submodule of $M$.
			The \emph{$S$-closure} or \emph{local closure} of $P$ in $M$ with respect to $S$ is defined as $P^S:=\varepsilon_{S,R,M}\inv(S\inv P)$.
	\end{itemize}
\end{definition}

Let $S$ be a left Ore set in a domain $R$.
In our paper \cite{HL_ISSAC17} we introduced the notion of left saturation closure of $S$, given by
\[
	\LSat(S)
	:=\set{r\in R\mid\exists~w\in R:wr\in S}.
\]
We proved that $\LSat(S)$ is a \emph{saturated} left Ore set in $R$ (i.e. for all $s,t\in R$ such that $s\cdot t\in \LSat(S)$ we have $s,t\in \LSat(S)$) and that $S\inv R$ and $\LSat(S)\inv R$ are isomorphic rings via $(s,r)\mapsto(s,r)$, which shows that $\LSat(S)$ is a canonical form of $S$ with respect to the corresponding localization.

To describe the $S$-closure more directly we introduce a notion of left saturation closure similar to the one for left Ore sets:

\begin{definition}\label{def_LSat}
	Let $S$ be a left denominator set in a ring $R$, $M$ a left $R$-module and $P$ a left $R$-submodule of $M$.
	The \emph{left saturation closure} of $P$ in $M$ with respect to $S$ is
	\[
		\LSat_S^M(P)
		:=\set{m\in M\mid\exists~s\in S:sm\in P}.
	\]
\end{definition}

Note that both notions of left saturation closures are instances of a more general concept which will be explored in a future paper.

\begin{lemma}
	Let $S$ be a left denominator set in a ring $R$, $M$ a left $R$-module and $P$ a left $R$-submodule of $M$.
	Then
	\[
		P^S=\LSat_S^M(P).
	\]
\end{lemma}
\begin{proof}
	Let $\varepsilon:=\varepsilon_{S,R,M}$.
	If $m\in P^S$, then $\varepsilon(m)\in S\inv P$, thus there exist $s\in S$ and $p\in P$ such that $(1,m)=\varepsilon(m)=(s,p)$.
	This implies the existence of $\tilde{s}\in S$ and $\tilde{r}\in R$ such that $\tilde{s}\cdot1=\tilde{r}s$ and $\tilde{s}m=\tilde{r}p\in P$, but the last equation implies $m\in\LSat_S^M(P)$.
	Now let $m\in\LSat_S^M(P)$, then there exists $s\in S$ such that $sm\in P$.
	But the $\varepsilon(m)=(1,m)=(s,sm)\in S\inv P$, thus $m\in\varepsilon\inv(S\inv P)=P^S$.
\end{proof}

\begin{lemma}\label{compatibility_of_closure_with_intersections}
	Let $S$ be a left denominator set in a ring $R$, $M$ a left $R$-module and $\set{P_j}_{j\in J}$ a family of left $R$-submodules of $M$.
	Consider their intersection $P:=\bigcap_{j\in J}^{}P_j$.
	\begin{enumerate}[(a)]
		\item
			We have $P^S\subseteq\bigcap_{j\in J}^{}P_j^S$.
		\item
			If $J$ is finite, then $P^S=\bigcap_{j\in J}^{}P_j^S$.
	\end{enumerate}
\end{lemma}
\begin{proof}
	\begin{enumerate}[(a)]
		\item
			Let $m\in P^S$, then there exists $s\in S$ such that $sm\in P=\bigcap_{j\in J}^{}P_j$, thus $sm\in P_j$ and $m\in P_j^S$ for all $j\in J$, which implies $m\in\bigcap_{j\in J}^{}P_j^S$.
		\item
			Let $m\in\bigcap_{j\in J}^{}P_j^s$, then for all $j\in J$ there exists $s_j\in S$ such that $s_jm\in P_j$.
			Since $J$ is finite there exists a common left multiple $s\in S$ of the $s_j$ by the left Ore condition, which implies $sm\in P_j$ for all $j\in J$.
			Therefore, $sm\in\bigcap_{j\in J}^{}P_j=P$ and $m\in P^S$.\qedhere
	\end{enumerate}
\end{proof}

\section{Algorithmic toolbox}

Let $K$ be a field and consider the two commutative polynomial rings $K[x]:=K[x_1,\ldots,x_n]$ and $K[y]:=K[y_1,\ldots,y_m]$ with the ideals $I=\leftideal{K[x]}{h_1,\ldots,h_k}$ and $J=\leftideal{K[y]}{g_1,\ldots,g_l}$.
Let further $\varphi:K[x]/I\rightarrow K[y]/J,~x_i\mapsto f_i$ be the ring map induced by elements $f_1,\ldots,f_n\in K[y]$.
\Cref{alg:kernel_for_polynomial_algebras} outlines a classical Gr\"obner-driven method for computing $\ker(\varphi)$ (for details see e.g. \cite{GPS08}, Section 1.8.10).

\begin{algorithm}
	\caption{\textsc{KernelPolynomialAlgebra}}
	\label{alg:kernel_for_polynomial_algebras}
	\KwIn{$K,I,J,\varphi$ as above.}
	\KwOut{$\ker(\varphi)$.}
	\Begin{
		$H:=\leftideal{K[x,y]}{h_1,\ldots,h_k,g_1,\ldots,g_l,x_1-f_1,\ldots,x_n-f_n}$\;
		compute $H':=H\cap K[x]$ by eliminating $y_1,\ldots,y_m$\;
		\Return{$H'$}\;
	}
\end{algorithm}

Given a homomorphism of arbitrary rings $\psi:A\rightarrow B$ and a two-sided ideal $J$ in $B$, we have that $\psi\inv(J)=\ker(\varphi)$ for the induced homomorphism $\varphi:A\rightarrow B/J$.
On the other hand the kernel of a homomorphism is the preimage of the zero ideal.
Therefore computing kernels and preimages of two-sided ideals is equivalent. Note that this does not hold for preimages of left or right ideals, see \cite{LVint}.

\section{Intersection of ideals with multiplicative sets in commutative polynomial algebras}

The first problem we are interested in solving is the following:

\begin{definition}
	Let $S$ be a left denominator set in a ring $R$ and $I$ a left ideal in $R$.
	The \textit{\textbf{intersection problem}} is to decide whether $I\cap S=\emptyset$ and to compute an element contained in this intersection whenever the answer is negative.
\end{definition}

In our paper \cite{HL_ISSAC17} we have shown that this problem is integral to a constructive treatment of the Ore condition in $G$-algebras which in turn allows us to perform basic arithmetic operations in Ore localizations of $G$-algebras.

In the commutative setting it is an important ingredient for solving linear systems over commutative localizations (\cite{Posur18}).

Here we consider commutative polynomial algebras of the form $R:=K[x]/J$, where $J$ is an ideal in the commutative polynomial ring $K[x]:=K[x_1,\ldots,x_n]$.
Furthermore, let $I$ be an ideal in $R$ and fix some suitable $g_i,h_i\in K[x]$ with $J=\leftideal{K[x]}{g_1,\ldots,g_\ell}$ and $I=\leftideal{R}{h_1+J,\ldots,h_k+J}$.
In the following we give algorithms to solve the intersection problem for $I\cap S$, where $S$ is a multiplicative subset of $R$ belonging to one of the localization types described in \Cref{classification_of_localization_types} with some computability restrictions.

\subsection{Monoidal}

In this subsection we start with algorithms in commutative rings and later proceed to non-commutative ones.

Suppose we are given a monoid $S\subseteq R$, finitely generated by a set $F=\{f_1 +J,\ldots, f_m+J\}$.
Then the monoid algebra $K[S]:=K[F]\subseteq R$ is a natural subalgebra of $R$.
Moreover, consider $\psi: K[t_1,\ldots,t_m]\to K[x]/J$, $t_i\mapsto f_i +J$, then the monoid algebra $K[S]$ is a finitely presented $K$-algebra which is isomorphic to $K[t]/\ker(\psi)$.
Since $R$ is commutative, but not necessarily a domain, we have to ensure that $S^{-1}R\neq\set{0}$, which is equivalent to $0\notin S$.
The latter property can be checked with \Cref{alg_ZeroContainedInMonoid}.

\begin{algorithm}
	\caption{\textsc{ZeroContainedInMonoid}}
	\label{alg_ZeroContainedInMonoid}
	\KwIn{A subset $F=\set{f_1+J,\ldots,f_m+J}\subseteq R=K[x]/J$.}
	\KwOut{$1$, if $0\in S=\merz{F}$, and $0$ otherwise.}
	\Begin{
		let $\psi: K[t_1,\ldots,t_m]\to K[x_1,\ldots,x_n]/J$, $t_i\mapsto f_i+J$\;
		$H:=\ker(\psi)$\tcp*{preimage $\psi^{-1}(0)$}
		$M:=H:\ideal{t_1\cdot\ldots\cdot t_m}^{\infty}$\;\label{0inS-line4}
		\uIf{$1\in M$}{
			\Return{$1$}\;
		}
		\Else{
			\Return{$0$}\;
		}
	}
\end{algorithm}

\begin{proposition}
	\Cref{alg_ZeroContainedInMonoid} terminates and is correct.
\end{proposition}
\begin{proof}
	We have $0\in S$ if and only if there exists $\alpha\in\IN_0^m$ such that $f^\alpha=f_1^{\alpha_1}\cdot\ldots\cdot f_m^{\alpha_m}\in J$, which in turn is equivalent to the existence of $\alpha\in\IN_0^m$ satisfying $t^\alpha\in\ker(\psi)=:H$.
	By e.~g. \cite{KR05, EM16} an ideal $H\subseteq K[t]$ contains a monomial if and only if the ideal $H:\erz{t_1\cdot\ldots\cdot t_m}^\infty$ contains $1$.
	Note that all operations involved are computable: the kernel $\ker(\psi)$ via \Cref{alg:kernel_for_polynomial_algebras} and the saturation $H:\erz{t_1\cdot\ldots\cdot t_m}^\infty$
	via \cite{GPS08}, Section 1.8.9.
\end{proof}

To solve the intersection problem in the monoidal case, we need to be able to determine the \emph{biggest monomial ideal} contained in an ideal in a commutative polynomial algebra, which can be computed with \Cref{alg_BiggestMonomialIdeal}.

\begin{algorithm}
	\caption{\textsc{BiggestMonomialIdeal}}
	\label{alg_BiggestMonomialIdeal}
	\KwIn{An ideal $L+J$ in $R=K[x]/J$.}
	\KwOut{The biggest monomial ideal contained in $L+J$.}
	\Begin{
		Let $K[x,q^{\pm1}]:=K[x,q_1,q_1\inv,\ldots,q_m,q_m\inv]$\;
		$\varphi:K[x]\rightarrow K[x,q^{\pm1}],~x_i\mapsto q_ix_i$ \tcp*{ring extension}
		$N:=\leftideal{K[x,q^{\pm 1}]}{\varphi(L)}\cap K[x]$ \tcp*{contraction of an ideal}
		\Return{$(N+J)/J$}\;
	}
\end{algorithm}

\begin{proposition}
	\Cref{alg_BiggestMonomialIdeal} terminates and is correct.
\end{proposition}
\begin{proof}
	Termination is clear.
	Consider the Laurent polynomial ring
	\[
		K[q^{\pm1}]:=K[q_1,q_1^{-1},\ldots,q_m,q_m^{-1}]
	\]
	and a homomorphism of $K$-algebras $\varphi:K[x]\to K[x,q^{\pm 1}]$, $x_i\mapsto q_i x_i$.
	By \cite{KR05, EM16, SST00}, the biggest monomial ideal contained in $L\subseteq K[x]$ is exactly $N$.
	
	Since for all $m\in K[x]$ we have $m+J\in L+J$ if and only if $m\in L$, this is in particular true for monomials.
	Therefore the biggest monomial ideal of $L+J$ is the biggest monomial ideal of $L$ modulo $J$.
\end{proof}

Now we have all the tools to consider the general situation.

\begin{algorithm}
	\caption{\textsc{NCIdealIntersectionWithMonoid}}
	\label{alg_NCIdealIntersectionWithMonoid}
	\KwIn{A left ideal $I\subseteq A$, a generating set (of a monoid $S$) $F=\{f_1,\ldots,f_m\}$ in the $K$-algebra $A$, such that $f_i\in A$ commute pairwise.}
	\KwOut{$I\cap S$: either $\emptyset$ or a finite set of monomial generators $\set{ t^{\alpha}:\alpha\in\IN_0^n}\subseteq[t_1,\ldots,t_m]$.}
	\Begin{
		$\psi:K[t_1,\ldots,t_m]\to A$, $t_i\mapsto f_i$\;
		$L:=\psi^{-1}(I)\subseteq K[t_1,\ldots,t_m]$\tcp*{preimage of $I\subseteq A$}
		\If{$\psi(L)=0$} {
			\Return{$\emptyset$}\tcp*{since then $\psi^{-1}(I) =\ker(\psi)$}
		}\label{preim_comp}
		$R := K[t_1,\ldots,t_m]/\ker(\psi)$\;
		$M := $ \textsc{BiggestMonomialIdeal}$(L,R)$\;
		\If{$M=\set{0}$}{
			\Return{$\emptyset$}\;
		}
		\Return{ $M$}\;
	}
\end{algorithm}

\begin{proposition}\label{proof_for_alg_NCIdealIntersectionWithMonoid}
	Let $A$ be an associative (but not necessarily commutative) unital $K$-algebra and $F=\{f_1,\ldots, f_m\}\subseteq A$ be a set of \textbf{pairwise commuting elements} in $A$.
	Moreover, let $S\subseteq A$ be the monoid in $A$ generated by $F$.
	Then \Cref{alg_NCIdealIntersectionWithMonoid} correctly computes $I\cap S$.
	Furthermore, its termination depends solely on the termination of the computation of $\psi^{-1}(I)$, which in turn depends on $A, I$ and $F$.
\end{proposition}
\begin{proof}
	The $K$-monoid algebra $K[S]=K[f_1,\ldots, f_m]\subseteq A$ is a $K$-subalgebra of $A$ and there is a natural homomorphism of $K$-algebras
	\[
		\psi: K[t_1,\ldots,t_m]\to A, t_i\mapsto f_i.
	\]
	Then $K[S]\cong K[t_1,\ldots,t_m]/\ker(\psi)$, hence the monoid algebra $K[S]$ is a finitely presented commutative $K$-algebra.
	As soon as the preimage $\psi^{-1}(I)=I\cap K[t_1,\ldots,t_m]$ is computable we are left with the following problem: given an ideal $L\subseteq K[t_1,\ldots,t_m]/J$, compute an intersection of $L$ with the submonoid $[t_1,\ldots,t_m]$, which is solved by \Cref{alg_BiggestMonomialIdeal}.
\end{proof}

\begin{corollary}
	Consider the situation of \Cref{proof_for_alg_NCIdealIntersectionWithMonoid}.
	\begin{itemize}
		\item
			If $A$ is a commutative polynomial algebra, \Cref{alg_NCIdealIntersectionWithMonoid} terminates for any $I$ and $F$.
		\item
			If $A$ is a $GR$-algebra, the \textsc{ncPreimage} algorithm from \cite{LVint} either returns the preimage or reports that the computability condition is violated.
			Namely, \textsc{ncPreimage} assumes that a $GR$-algebra $A$ is equipped with an \emph{admissible} elimination ordering.
			If a certain integer programming problem has a solution, such an ordering can be constructed from it, while infeasibility of the problem proves that no such ordering exists.
	\end{itemize}
\end{corollary}

\subsection{Geometric}

Let $\mf{p}=\leftideal{R}{p_1+J,\ldots,p_m+J}$ be a prime ideal in $R$ with $p_i\in K[x]$ and consider the multiplicative set $S:=R\setminus\mf{p}$.
The preimage of $\mf{p}$ under the canonical surjection $K[x]\rightarrow R$ is given by the ideal $\mf{q}:=\leftideal{K[x]}{p_1,\ldots,p_m,g_1,\ldots,g_\ell}$.
Now there are two possible cases:
\begin{description}
	\item[Case 1:]
		$h_i\in\mf{q}$ for all $i$.
		Then $h_i+J\in\mf{p}$ for all $i$ and thus $I\subseteq\mf{p}$, which implies $I\cap S=I\cap(R\setminus\mf{p})=I\setminus\mf{p}=\emptyset$.
	\item[Case 2:]
		$h_i\notin\mf{q}$ for some $i$.
		Then $h_i+J\notin\mf{p}$ and thus $I\cap S\neq\emptyset$.
\end{description}
Since ideal membership in polynomial rings can be decided with Gr\"obner basis tools, this observations lead to \Cref{alg:geometric_intersection}, where $\NF(h_i|\mf{q})$ denotes the normal form of $h_i$ with respect to (a Gr\"obner basis of) the ideal $\mf{q}$.

\begin{algorithm}
	\caption{\textsc{CommutativeGeometricIntersection}}
	\label{alg:geometric_intersection}
	\KwIn{Ideals $I,J,\mf{p}$ and the multiplicative set $S$ as above.}
	\KwOut{An element of $I\cap S$ (if $I\cap S\neq\emptyset$) or $0$ (if $I\cap S=\emptyset$).}
	\Begin{
		let $\mf{q}:=\leftideal{K[x]}{p_1,\ldots,p_m,g_1,\ldots,g_\ell}\subseteq K[x]$\;
		\ForEach{$i\in\nset{k}$}{
			\If{$\NF(h_i|\mf{q})\neq0$}{
				\Return{$h_i+J$}\;
			}
		}
		\Return{$0$}\;
	}
	
\end{algorithm}

\subsection{Rational}

Let $r\in\nset{n}$ and consider $\hat{S}:=K[x_1+J,\ldots,x_r+J]$ as well as $S:=\hat{S}\setminus\set{0}$.
Let $K[t]:=K[t_1,\ldots,t_r]$ and define the map $\varphi:K[t]\rightarrow R$ by $t_i\mapsto x_i$ for $1\leq i\leq r$.

\begin{lemma}\label{rational_intersection_characterization}
	In the situation above we have $I\cap S=\emptyset$ if and only if $\varphi\inv(I)\subseteq\ker(\varphi)$.
\end{lemma}
\begin{proof}
	Let $I\cap S=\emptyset$, then $I\cap\hat{S}=\set{0}$.
	Now
	$\varphi\inv(I)\subseteq\ker(\varphi)$ follows directly from
	$\varphi(\varphi\inv(I))\subseteq I\cap\im(\varphi)=I\cap\hat{S}=\set{0}$.

	On the other hand, let $\varphi\inv(I)\subseteq\ker(\varphi)$ and choose an element $w\in I\cap\hat{S}=I\cap\im(\varphi)$.
	Then there exists $v\in K[t]$ such that $\varphi(v)=w\in I$ and thus $v\in\varphi\inv(I)\subseteq\ker(\varphi)$.
	This implies $w=\varphi(v)=0$ and therefore $I\cap\hat{S}=\set{0}$ or, equivalently, $I\cap S=\emptyset$.
\end{proof}

\begin{algorithm}
	\caption{\textsc{CommutativeRationalIntersection}}
	\label{alg:rational_intersection}
	\KwIn{$I,J,r,S$ as above.}
	\KwOut{An element of $I\cap S$ (if $I\cap S\neq\emptyset$) or $0$ (if $I\cap S=\emptyset$).}
	\Begin{
		let $\varphi:K[t]\rightarrow K[x]/J,~t_i\mapsto x_i$\;
		compute the preimage $\varphi\inv(I)=\leftideal{K[t]}{w_1,\ldots,w_m}$\;
		\ForEach{$i\in\nset{m}$}{
			\If{$\varphi(w_i)\neq0$}{
				\Return{$\varphi(w_i)$}\;
			}
		}
		\Return{$0$}\;
	}
\end{algorithm}

\begin{proposition}
	\Cref{alg:rational_intersection} terminates and is correct.
\end{proposition}
\begin{proof}
	Termination is obvious.
	The preimage $\varphi\inv(I)$ can be computed via \Cref{alg:kernel_for_polynomial_algebras}.
	Now we check whether $\varphi\inv(I)$ is contained in $\ker(\varphi)$ on the generators $w_i$.
	Correctness follows then from \Cref{rational_intersection_characterization}.
\end{proof}

Another way to look at \Cref{alg:rational_intersection} is the following: the preimage computation gives us \\
$\leftideal{K[x]}{g_1,\ldots,g_\ell,h_1,\ldots,h_k}\cap K[x_1,\ldots,x_r]$; then we look for generators of this ideal that do not vanish modulo $J$.

\section{Applications to local closure in the commutative setting}
\label{sectAppLocalClosureComm}

Recall the following basic concepts of the theory of commutative rings:
The \emph{radical} of an ideal $I$ in a commutative ring $R$ is the ideal defined as $\sqrt{I}:=\set{r\in R\mid r^n\in I\text{ for some }n\in\IN}$.
A proper ideal $I$ of a commutative ring $R$ is called \emph{primary} if for all $a,b\in R$ such that $ab\in I$ we have $a\in I$ or $b\in\sqrt{I}$.
In a more symmetric view, $I$ is primary if and only if for all $a,b\in R$ with $ab\in I$ we have $a\in I$ or $b\in I$ or ($a\in\sqrt{I}$ and $b\in\sqrt{I}$).
The radical of a primary ideal is always a prime ideal.
If $Q$ is primary with radical $P=\sqrt{Q}$, then $Q$ is also called \emph{$P$-primary}.
An ideal is called \emph{decomposable} if it can be written as an intersection of finitely many primary ideals.

The goal of this section is to show how to compute the $S$-closure of an ideal $I$ in a commutative ring $R$ under two assumptions:
\begin{enumerate}[(1)]
	\item
		We can decide whether $Q\cap S=\emptyset$ for any primary ideal $Q$ in $R$.
	\item
		The ideal $I$ is decomposable and we are either given a primary ideal decomposition of $I$ or are able to compute one.
		Note that a primary decomposition always exists in Noetherian rings.
		In polynomial algebras it can be algorithmically computed, see e.g. \cite{GPS08}, though its embedded components are not unique.
\end{enumerate}

The main ingredient is the following observation, which highlights the differences between primary ideals and arbitrary ideals:

\begin{lemma}\label{closure_of_primary_ideals}
	Let $S$ be a multiplicative set of a commutative ring $R$.
	\begin{enumerate}[(a)]
		\item
			If $I$ is an arbitrary ideal in $R$ such that $I\cap S\neq\emptyset$, then $I^S=R$.
		\item
			If $Q$ is a primary ideal in $R$ such that $Q\cap S=\emptyset$, then $Q^S=Q$.
	\end{enumerate}
\end{lemma}
\begin{proof}
	Let $w\in I\cap S$, then $w\cdot1=w\in I$, thus $1\in I^S$ and therefore $I^S=R$.
	On the other hand, let $Q\cap S=\emptyset$ and $r\in Q^S$, then there exists $s\in S$ such that $sr\in Q$.
	Since $Q$ is primary we have $s\in Q$ or $r\in Q$ or ($s\in\sqrt{Q}$ and $r\in\sqrt{Q}$).
	But $s\notin\sqrt{Q}$ (and therefore $s\notin Q$), because otherwise $s^n\in Q\cap S=\emptyset$ for some $n\in\IN$.
	Thus the only remaining option is $r\in Q$, which implies $Q^S=Q$.
\end{proof}

Let $I$ be a decomposable ideal with a primary decomposition $I=\bigcap_{i=1}^{n}Q_i$.
Then $I^S=\bigcap_{i=1}^{n}Q_i^S$ by \Cref{compatibility_of_closure_with_intersections}.
Combining this with \Cref{closure_of_primary_ideals} we can compute $I^S$ for any multiplicative set $S$ via \Cref{alg:decomposition_closure} if we can decide non-emptiness of the intersections $Q_i\cap S$.

\begin{algorithm}
	\caption{\textsc{CommutativeLocalClosureDecomp}}
	\label{alg:decomposition_closure}
	\KwIn{A decomposable ideal $I=\bigcap_{i=1}^{n}Q_i$ and a multiplicative set $S$ in a commutative ring $R$.}
	\KwOut{$I^S$.}
	\Begin{
		\ForEach{$i\in\nset{n}$}{
			\uIf{$Q_i\cap S=\emptyset$}{
				$\tilde{Q}_i:=Q_i$\;
			}
			\Else{
				$\tilde{Q}_i:=R$\;
			}
		}
		\Return{$\tilde{I}:=\bigcap_{i=1}^{n}\tilde{Q}_i$}\;
	}
\end{algorithm}

In particular, we have the following:

\begin{corollary}
	Let $S$ be a multiplicative set of a commutative ring $R$ and $I$ a decomposable ideal in $R$.
	Then there exists an ideal $J$ in $R$ satisfying $I=I^S\cap J$ and $S\inv J=R$.
\end{corollary}
\begin{proof}
	Let $I=\bigcap_{i=1}^{n}Q_i$ be a decomposition of $I$ into primary ideals.
	Then $I^S$ is the intersection of all $Q_i$ such that $Q_i\cap S=\emptyset$.
	Define $J$ to be the intersection of all $Q_i$ such that $Q_i\cap S\neq\emptyset$, then the claim follows from the observations above.
\end{proof}

The question remains how to decide whether $Q\cap S$ is empty or not.
This can be reduced to the same question for prime ideals:

\begin{lemma}
\label{Lemma20}
	Let $S$ be a multiplicative set of a commutative ring $R$ and $Q$ a $P$-primary ideal in $R$.
	Then $Q\cap S=\emptyset$ if and only if $P\cap S=\emptyset$.
\end{lemma}
\begin{proof}
	If $P\cap S=\emptyset$, then $Q\cap S=\emptyset$ since $Q\subseteq P$.
	If $s\in P\cap S$, then $s^n\in Q\cap S$ for some $n\in\IN$.
\end{proof}

For certain localization types, the latter can be characterized in a way that allows for an algorithmic treatment:

\begin{lemma}
\label{Lemma21}
	Let $S$ be a multiplicative set of a commutative ring $R$ and $P$ a prime ideal in $R$.
	\begin{enumerate}[(a)]
		\item
			Let $S=\merz{s_1,\ldots,s_n}$ for some $s_i\in R$.
			Then $P\cap S=\emptyset$ if and only if $P\cap\set{s_1,\ldots,s_n}=\emptyset$.
		\item
			Let $S=R\setminus\mf{p}$ \; for some prime ideal $\mf{p}\subset R$.
			Then $P\cap S=\emptyset$ if and only if $P\subseteq\mf{p}$.
		\item
			Let $S=T\setminus\set{0}$ for some subring $T$ of $R$.
			Then $P\cap S=\emptyset$ if and only if $P\cap T=\set{0}$.
	\end{enumerate}
\end{lemma}
\begin{proof}
	The last two claims are obvious from the definitions, so assume $S=\merz{s_1,\ldots,s_n}$.
	If $P\cap S=\emptyset$, then $P\cap\set{s_1,\ldots,s_n}=\emptyset$ since $\set{s_1,\ldots,s_n}\subseteq S$.
	If $s\in P\cap S$, then $s=\prod_{i=1}^{n}s_i^{\nu_i}\in P$ for some $\nu_i\in\IN_0$, where for at least one $j\in\nset{n}$ we have $\nu_j\geq1$ since $1\notin P\cap S$.
	But then $s_j\in P$ since $P$ is prime.
\end{proof}

In polynomial algebras, this enables computations, since ideal membership test and intersection with essential subalgebras  are classical applications of Gr\"obner bases.

With the methods developed so far we can address the important notion of the \myemph{symbolic power} of an ideal (see \cite{SymbPower17} for a modern overview of a vivid area of investigations). Let $I\subseteq R$ be an ideal in a Noetherian domain $R$. Suppose that $\displaystyle I = \bigcap_{i=1}^r Q_i$ is a primary decomposition with associated primes $\mf{p}_i := \sqrt{Q_i}$. Then the $n$-th symbolic power of $I$ is defined to be
\[
I^{(n)}=\bigcap_{i=1}^{r} \;  (R_{\mf{p}_i} I^n \cap R),
\]
where  $R_{\mf{p}_i} = (R\setminus{\mf{p}_i})^{-1} R$ is a  localization of geometric type with respect to the Ore set $S_i := R\setminus \mf{p}_i$. Furthermore, $R_{\mf{p}_i} I^n \cap R$ is exactly the local closure of $I^n$ with respect to $S_i$, which implies the inclusion $I^n\subseteq I^{(n)}$ for all $n$. 
 In \Cref{ex:symbolic_power} we will see that equality does not hold in general.
 In the special case where $I = \mf{p}$ is a prime ideal, the symbolic power $\mf{p}^{(n)}$ is precisely the $\mf{p}$-primary component of $\mf{p}^n$. 

These observations immediately lead to the following \Cref{alg:symbolic_power} as an application of \Cref{alg:decomposition_closure} and utilizing \Cref{alg:geometric_intersection} to decide non-emptiness of the intersections.

\begin{algorithm}
	\caption{\textsc{SymbolicPower}}
	\label{alg:symbolic_power}
	\KwIn{A decomposable ideal $I$ in a commutative domain $R$ and $n\in\IN$.}
	\KwOut{$I^{(n)}$.}
	\Begin{
		compute the associated primes $\mf{p}_1, \ldots, \mf{p}_r$ of $I$\;
		compute a primary decomposition $I^n=\bigcap_{i=1}^{m} Q_i$ with associated primes $\mf{q}_i:=\sqrt{Q_i}$\;
		\ForEach{$i\in\nset{m}$}{
			$\tilde{Q}_i:=R$\;
			\ForEach{$j\in\nset{r}$}{
				\uIf{$\mf{q}_i \subseteq \mf{p}_j$\label{sympower-ideal-cont}}{
				$\tilde{Q}_i:= \tilde{Q}_i \cap Q_i$\;
				}
			}
		}
		\Return{$\tilde{I}:=\bigcap_{i=1}^{m}\tilde{Q}_i$}\;
	}
\end{algorithm}

\begin{proposition}
	\Cref{alg:symbolic_power} terminates and is correct.
\end{proposition}
\begin{proof}
	Termination is obvious. The containment of ideals in \Cref{sympower-ideal-cont} is equivalent to $Q_i\cap S_j=\emptyset$ via \Cref{Lemma20} and \Cref{Lemma21}.
	If the latter condition fails, $Q_i$ can be ignored, since $Q_{i}^{S_j}=R$ by \Cref{closure_of_primary_ideals}. 
\end{proof}

\begin{example}\label{ex:symbolic_power}
	Perhaps the most popular example for $\mf{p}^{(2)}\neq \mf{p}^2$ is given by
	\[
		\mf{p}
		=\erz{x^4-yz, y^2-xz, x^3y-z^2}\subseteq\mathbb{Q}[x,y,z].
	\]
	In this situation, $\mf{p}^2$ has two associated primes: $\mf{p}$ and $\erz{x,y,z}$.
	Only the $\mf{p}$-primary component survives in the closure, giving $\mf{p}^{2}=\mf{p}^{(2)}\cap\erz{z,y^4,x^3 y^3,x^4 y^2,x^7y,x^8}$, where
	\[
		\mf{p}^{(2)}
		=\erz{y^{4}-2xy^{2}z+x^{2}z^{2}, x^{3}y^{3}-x^{4}yz-y^{2}z^{2}+xz^{3}, x^{4}y^{2}-x^{5}z-y^{3}z+xyz^{2}, x^{7}+x^{2}y^{3}-3x^{3}yz+z^{3}}.
	\]
\end{example}

\section{Central closure of submodules}

A \emph{central closure of a submodule} is the local closure of a submodule with respect to a left Ore set which is contained in the center of the underlying ring.
The goal of this section is to develop algorithms for certain central closures of submodules over $G$-algebras.

\subsection{The class of G-algebras}

Recall that a total ordering $\leq$ on $\IN_0^n$ with least element $0$ is called \emph{admissible} if $\alpha\leq\beta$ implies $\alpha+\gamma\leq\beta+\gamma$ for all $\alpha,\beta,\gamma\in\IN_0^n$.

\begin{definition}
	Let $K$ be a field and $A$ a $K$-algebra generated by $x_1,\ldots,x_n$.
	\begin{itemize}
		\item
			The set of \emph{standard monomials} of $A$ is
			\[
				\Mon(A)
				:=\set{x^\alpha\mid\alpha\in\IN_0^n}
				:=\set{x_1^{\alpha_1}x_2^{\alpha_2}\cdots x_n^{\alpha_n}\mid\alpha_i\in\IN_0}.
			\]
		\item
			Let $\leq$ be an admissible ordering on $\IN_0^n$.
			Any $f\in\leftideal{K}{\Mon(A)}\setminus\set{0}$ has a unique representation $f=\sum_{\alpha\in\IN_0^n}^{}c_\alpha x^\alpha$ for some $c_\alpha\in K$, where $c_\alpha=0$ for almost all $\alpha$, but $c_\alpha\neq0$ for at least one $\alpha$.
			Now we define
			\begin{itemize}
				\item
					$\lexp_\leq(f):=\max_\leq\set{\alpha\in\IN_0^n\mid c_\alpha\neq0}$, the \emph{leading exponent} of $f$ with respect to $\leq$,
				\item
					$\lc_\leq(f):=c_{\lexp_\leq(f)}\in K\setminus\set{0}$, the \emph{leading coefficient} of $f$ with respect to $\leq$,
				\item
					$\lm_\leq(f):=x^{\lexp_\leq(f)}\in\Mon(A)$, the \emph{leading monomial} of $f$ with respect to $\leq$.
			\end{itemize}
	\end{itemize}
\end{definition}

\begin{definition}
	For $n\in\IN$ and $1\leq i<j\leq n$ consider the constants $c_{ij} \in K\setminus\set{0}$ and polynomials $d_{ij}\in K[x_1,\ldots,x_n]$.
	Suppose that there exists an admissible ordering $\leq$ on $\IN_0^n$ such that for any $1\leq i<j\leq n$ either $d_{ij}=0$ or $\lexp_\leq(d_{ij})<\lexp_\leq(x_ix_j)$.
	The $K$-algebra
	\[
		A
		:=K\erz{x_1,\ldots,x_n\mid\set{ x_j x_i = c_{ij} x_i x_j + d_{ij}:1\leq i<j\leq n}}
	\]
	is called a \emph{$G$-algebra} if $\Mon(A)$ is a $K$-basis of $A$.
\end{definition}

$G$-algebras (\cite{LS, lev_diss}) are also known as algebras of solvable type (\cite{KW,Kr,Kredel2015}) and as PBW algebras (\cite{BGV}).
$G$-algebras are left and right Noetherian domains that occur naturally in various situations and encompass algebras of linear functional operators modeling difference and differential equations.

\begin{example}
	Let $K$ be a field, $q_i\in K\setminus\set{0}$ and $n\in\IN$.
	Common $G$-algebras include the following examples, where only the relations between non-commutating variables are listed:
	\begin{itemize}
		\item
			The commutative polynomial ring $K[x_1,\ldots,x_n]$.
		\item
			The $n$-th \emph{Weyl algebra} $A_n:=K\erz{x_1,\ldots,x_n,\partial_1,\ldots,\partial_n}$ with $\partial_ix_i=x_i\partial_i+1$ for all $1\leq i\leq n$.
		\item
			The $n$-th \emph{shift algebra} $S_n:=K\erz{x_1,\ldots,x_n,s_1,\ldots,s_n}$ with $s_ix_i=x_is_i+s_i=(x_i+1)s_i$ for all $1\leq i\leq n$.
		\item
			The $n$-th \emph{$q$-shift algebra} $S^{(q)}_n:=K\erz{x_1,\ldots,x_n,s_1,\ldots,s_n}$ with $s_ix_i=q_i x_is_i$ for all $1\leq i\leq n$.
		\item
			The $n$-th \emph{$q$-Weyl algebra} $A^{(q)}_n:=K\erz{x_1,\ldots,x_n,\partial_1,\ldots,\partial_n}$ with $\partial_ix_i=q_i x_i\partial_i+1$ for all $1\leq i\leq n$.
		\item
			The $n$-th \emph{integration algebra} $K\erz{x_1,\ldots, x_n, I_1,\ldots, I_n}$ with $I_i x_i = x_i I_i + I_i^2$  for all $1\leq i\leq n$.
	\end{itemize}
\end{example}

Furthermore, there exists a well-developed Gr\"obner basis theory for $G$-algebras which is close to the commutative case.

If considered over a field of prime characteristic, Weyl and shift algebras have very big centers.
A similar situation happens if all quantum parameters $q_i$ are roots of unity for $q$-shift and $q$-Weyl algebras.
Then the mentioned algebras are finitely generated modules over their centers and thus enjoy a lot of commutativity.

Note that any admissible ordering $\leq$ on $\IN_0^n$ can be extended to an admissible ordering $\preceq$ on $\nset{r}\times\IN_0^n$, e.g. $\preceq=\POTify{\leq}$:

\begin{definition}
	Let $\leq$ be an admissible ordering on $\IN_0^n$.
	The \emph{(ascending) position-over-term ordering} extending $\leq$ is the ordering $\POTify{\leq}$ on $\nset{r}\times\IN_0^n$, defined via
	\[
		(i,\alpha)\POTify{\leq}(j,\beta)
		\quad:\Leftrightarrow\quad
		i<j\text{ or }(i=j\text{ and }\alpha\leq\beta)
	\]
	for $\alpha,\beta\in\IN_0^n$ and $i,j\in\nset{r}$.
\end{definition}

We recall the definition of Gr\"obner bases and their characterization via left normal forms in preparation for the forthcoming \Cref{alg:CentralEssentialRationalClosure}.

\begin{definition}\label{definition_of_Gr\"obner_bases}
	Let $A$ be a $G$-algebra, $r\in\IN$, $I$ a left $A$-submodule of $A^r$, $G$ a finite subset of $I$ and $\preceq$ an admissible ordering on $\nset{r}\times\IN_0^n$.
	Then $G$ is a \emph{left Gr\"obner basis} of $I$ with respect to $\preceq$ if for all $f\in I\setminus\set{0}$ there exists $g\in G$ such that $\lm_\preceq(g)\mid\lm_\preceq(f)$.
\end{definition}

\begin{algorithm}
	\caption{\textsc{LeftNF}}
	\label{alg_LeftNF}
	\KwIn{A $G$-algebra $A$ generated by variables $x=\set{x_1,\ldots,x_n}$, $r\in\IN$, $f\in A^r$, a finite subset $G$ of $A^r$, an admissible ordering $\preceq$ on $\nset{r}\times\IN_0^n$.}
	\KwOut{$\LeftNF_\preceq(f|G)$.}
	\Begin{
		$h:=f$\;
		\While{$h\neq0$
		and $G_h:=\set{g\in G:\lm_\preceq(g)\mid\lm_\preceq(h)}\neq\emptyset$}{
			choose $g\in G_h$\;
			$(i,\alpha):=\lexp_\preceq(h)$\;
			$(i,\beta):=\lexp_\preceq(g)$\;
			$h:=h-\frac{\lc_\preceq(h)}{\lc_\preceq(x^{\alpha-\beta}g)}x^{\alpha-\beta}g$\;
		}
		\Return{$h$}\;
	}
\end{algorithm}

\begin{theorem}\label{characterization_of_Gr\"obner_bases_in_terms_of_NF}
	In the situation of \Cref{definition_of_Gr\"obner_bases} the following are equivalent:
	\begin{enumerate}[(1)]
		\item
			$G$ is a left Gr\"obner basis of $I$ with respect to $\preceq$.
		\item
			$\LeftNF_{\preceq}(f|G)=0$ for all $f\in I$.
	\end{enumerate}
\end{theorem}

Further information can be found in e.~g. \cite{lev_diss} and \cite{BGV}.

\subsection{Central saturation}

Denote the center of a ring $R$ by $\central{R}$.
Its elements are called \emph{central}.

\begin{definition}\label{def_central_saturation}
	Let $R$ be a ring, $q\in\central{R}$, $k\in\IN$ and $I$ a left $R$-submodule of $R^k$.
	\begin{itemize}
		\item
			The \emph{(central) quotient} of $I$ by $q$ is the left $R$-submodule
			\[
				I:q
				:=\set{f\in R^k\mid qf\in I}
				=\set{f\in R^k\mid fq\in I}.
			\]
		\item
			The \emph{central saturation} of $I$ by $q$ is the left $R$-submodule
			\[
				I:q^\infty
				:=\bigcup_{i\in\IN_0}(I:q^i)
				=\set{f\in R^k\mid\exists~n\in\IN_0:q^nf\in I}.
			\]
		\item
			The \emph{(central) saturation index} of $I$ by $q$ is
			\[
				\operatorname{Satindex}(I,q)
				:=\min(\set{n\in\IN_0\mid(I:q^\infty)=(I:q^n)}\cup\set{\infty}).
			\]
	\end{itemize}
\end{definition}

These saturations themselves are special cases of (generalized) left saturation closures, since $I:q=\LSat_{\set{q}}^{R^k}(I)$ and $I:q^\infty=\LSat_{\merz{q}}^{R^k}(I)=I^{\merz{q}}$ (by extending \Cref{def_LSat} to arbitrary sets $S$).

\begin{remark}
	In the situation of \Cref{def_central_saturation}, consider the left $R$-module homomorphism $\phi:R^k\rightarrow R^k/I,~f\mapsto fq+I$.
	We have
	\[
		\ker(\phi)
		=\set{f\in R^k\mid qf+ I=\phi(f)=0+I}
		=I:q.
	\]
	Thus, if we can compute kernels of such left $R$-module homomorphisms, we can also compute central quotients.
	Furthermore, if we can decide equality of left $R$-modules, then we can also compute the central saturation iteratively, provided the saturation index is finite.
	The latter is always the case for Noetherian rings.
\end{remark}

In particular, we have the following result for finitely generated monoidal central closures:

\begin{remark}\label{monoidal_closure_in_G-algebras}
	Let $S=\merz{f_1,\ldots,f_k}$ be a left Ore set in a $G$-algebra $A$, $I$ a left ideal in $A$ and $z\in\central{A}\cap S$.
	Then $I^{\merz{z}}=\LSat_{\merz{q}}(I)=I:q^\infty$ is computable.
	Since $\merz{z}\subseteq S$ we have $I^{\merz{z}}\subseteq I^S$.
	The other inclusion holds if $\LSat(S)=\LSat(\merz{z})$, which is equivalent to $f_j\in\LSat(\merz{z})$ for all $j$.
	A sufficient condition for this is that $f_1,\ldots,f_k$ commute pairwise and $z$ is a multiple of $f_1\cdot\ldots\cdot f_k$.
	This also includes the special case where $f_1,\ldots,f_k\in\central{A}$.
\end{remark}

The restriction to central $q$ allows straightforward generalizations of classical commutative results, e.g. regarding the decomposition of ideals:

\begin{lemma}\label{saturation_leads_to_decomposition}
	Let $I$ be a left ideal in a ring $R$ and $q\in\central{R}$.
	If $n:=\operatorname{Satindex}(I,q)<\infty$, then $I=\leftideal{R}{I,q^n}\cap(I:q^n)$.
\end{lemma}
\begin{proof}
	Let $J:=\leftideal{R}{I,q^n}\cap(I:q^n)$.
	Since $I\subseteq\leftideal{R}{I,q^n}$ and $I\subseteq(I:q^n)$ we clearly have $I\subseteq J$.
	On the other hand, let $a\in J$, then $q^na\in I$ (since $a\in(I:q^n)$) and $a=b+rq^n$ for some $b\in I$ and $r\in R$ (since $a\in\leftideal{R}{I,q^n}$).
	Now
	\[
		q^{2n}r
		=q^nrq^n
		=q^n(a-b)
		=q^na-q^nb
		\in I
	\]
	shows that $r\in(I:q^{2n})=(I:q^n)$, which implies $rq^n=q^nr\in I$, thus $a=b+rq^n\in I$.
\end{proof}

\subsection{Antiblock orderings}

In preparation for the upcoming \Cref{alg:CentralEssentialRationalClosure} we need the notion of \emph{antiblock orderings} as well as some basic results which are included here for the sake of completeness.

Let $n,m,r\in\IN$.

\begin{definition}
	Let $\leq$ be an admissible ordering on $\IN_0^n$.
	The \emph{(ascending) position-over-term ordering} extending $\leq$ is the ordering $\POTify{\leq}$ on $\nset{r}\times\IN_0^n$, defined via
	\[
		(i,\alpha)\POTify{\leq}(j,\beta)
		\quad:\Leftrightarrow\quad
		i<j\text{ or }(i=j\text{ and }\alpha\leq\beta)
	\]
	for $\alpha,\beta\in\IN_0^n$ and $i,j\in\nset{r}$.
\end{definition}

\begin{lemma}\label{ordering_POT_induces_ordering_in_second_block}
	Let ${\leq}=(\leq_1,\leq_2)$ be an $(n,m)$-antiblock ordering on $\IN_0^{n+m}\cong\IN_0^n\times\IN_0^m$ and ${\preceq}:=\POTify{\leq}$.
	Let $(\alpha_1,\alpha_2),(\beta_1,\beta_2)\in\IN_0^{n+m}$ and $i,j\in\nset{r}$ such that we have $(i,(\alpha_1,\alpha_2))\preceq(j,(\beta_1,\beta_2))$, then $(i,\alpha_2)\POTify{\leq_2}(j,\beta_2)$.
\end{lemma}

\begin{proof}
	We have
	\[\begin{split}
		\quad(i,(\alpha_1,\alpha_2))\preceq(j,(\beta_1,\beta_2))
		\quad
		\Leftrightarrow&\quad(i,(\alpha_1,\alpha_2))\POTify{\leq}(j,(\beta_1,\beta_2))\\
		\Leftrightarrow&\quad i<j\text{ or }(i=j\text{ and }(\alpha_1,\alpha_2)\leq(\beta_1,\beta_2))\\
		\Rightarrow&\quad i<j\text{ or }(i=j\text{ and }\alpha_2\leq_2\beta_2)\\
		\Leftrightarrow&\quad(i,\alpha_2)\POTify{\leq_2}(j,\beta_2).\qedhere
	\end{split}\]
\end{proof}

\subsection{Central essential rational closure}\label{CERC_setting}

In this section, let $K$ be a field, $n,m,r\in\IN$ and $A$ a $G$-algebra over $K$ generated by two blocks of variables $x=\set{x_1,\ldots,x_n}$ and $y=\set{y_1,\ldots,y_m}$ such that $x$ generates a sub-$G$-algebra $B$ of $A$ with $B\subseteq\central{A}$.
Then $S:=B\setminus\set{0}$ is a left Ore set in $B$ as well as in $A$ since it is a multiplicative set consisting of central elements.
Furthermore, let ${\leq}=(\leq_1,\leq_2)$ be an $(n,m)$-antiblock ordering satisfying the ordering condition for $G$-algebras on $A$ and ${\preceq}:=\POTify{\leq}$.
Finally, let $\varepsilon:=\varepsilon_{S,A,A^r}$, $\rho:=\rho_{S,A}$ and ${\preceq_2}=\POTify{\leq_2}$.
Observe the following:

\begin{itemize}
	\item
		We have $\ker(\varepsilon)=\set{m\in A^r\mid\exists~s\in S:sm=0}=\set{0}$ since $A$ is a domain, thus $\varepsilon$ is injective.
	\item
		Since $B\subseteq\central{A}$ we can identify the subring $B$ with the commutative polynomial ring $K[x]=K[x_1,\ldots,x_n]$.
		Then we have $S\inv B\cong K(x)$.
	\item
		We can view $S\inv A$ as a $G$-algebra over the field $K(x)$ in the variables $y_1,\ldots,y_m$ with the relations inherited from $A$, thus the Gr\"obner basis theory of $G$-algebras applies.
	\item
		The monomials in the module $S\inv(A^r)\cong(S\inv A)^r$ are of the form $\varepsilon(y^\alpha e_i)=\rho(y^\alpha)\varepsilon(e_i)$.
		Let $(s,f)\in S\inv A^r$, then $(s,f)$ and $(1,f)=\varepsilon(f)$ have the same leading exponent and the same leading monomial with respect to $\preceq_2$.
\end{itemize}

\begin{lemma}\label{ordering_POT_inheritance_of_leading_exponent_in_second_block}
	Let $f\in A^r\setminus\set{0}$ and $\lexp_{\preceq}(f)=(i,(\alpha_1,\alpha_2))$, then $\lexp_{\preceq_2}(\varepsilon(f))=(i,\alpha_2)$.
\end{lemma}
\begin{proof}
	Define $\compset{r}:=\set{1,\ldots,r}$ and let
	\[
		f
		=\sum_{(j,(\beta,\gamma))\in\ModExp{r}{n+m}}^{}c_{(j,(\beta,\gamma))}x^\beta y^\gamma e_j
		=\sum_{(j,\gamma)\in\ModExp{r}{m}}^{}\underbrace{\left(\sum_{\beta\in\IN_0^n}^{}c_{(j,(\beta,\gamma))}x^\beta\right)}_{\eqqcolon\tilde{c}_{(j,\gamma)}\in K[x]}y^\gamma e_j
	\]
	with $c_{(j,(\beta,\gamma))}\in K$, then $(j,(\beta,\gamma))\preceq(i,(\alpha_1,\alpha_2))=\lexp_\preceq(f)$ whenever $c_{(j,(\beta,\gamma))}\neq0$.
	Furthermore,
	\[
		\varepsilon(f)
		=\varepsilon\left(\sum_{(j,\gamma)\in\ModExp{r}{m}}^{}\tilde{c}_{(j,\gamma)}y^\gamma e_j\right)
		=\sum_{(j,\gamma)\in\ModExp{r}{m}}^{}\varepsilon(\tilde{c}_{(j,\gamma)}y^\gamma e_j)
		=\sum_{(j,\gamma)\in\ModExp{r}{m}}^{}\rho(\tilde{c}_{j,\gamma})\cdot\varepsilon(y^\gamma e_j)
	\]
	implies that it suffices to show that $(j,\gamma)\preceq_2(i,\alpha_2)$ whenever $\tilde{c}_{(j,\gamma)}\neq0$.
	The last condition implies that there is some $\beta\in\IN_0^n$ such that $c_{(j,(\beta,\gamma))}\neq0$.
	Now $(j,(\beta,\gamma))\preceq(i,(\alpha_1,\alpha_2))=\lexp_\preceq(f)$ implies $(j,\gamma)\preceq_2(i,\alpha_2)$ by \Cref{ordering_POT_induces_ordering_in_second_block}, thus $\lexp_{\preceq_2}(\varepsilon(f))=(i,\alpha_2)$.
\end{proof}

\begin{proposition}\label{ordering_inheritance_of_Gr\"obner_basis_property}
	Let $I$ be a left $A$-submodule of $A^r$ and $G$ a left Gr\"obner basis of $I$ with respect to $\preceq$.
	Then $\varepsilon(G)$ is a left Gr\"obner basis of $J:=S\inv I$ with respect to ${\preceq_2}=\POTify{\leq_2}$.
\end{proposition}
\begin{proof}
	Let $z\in J\setminus\set{0}$, then $z=(s,f)$ for some $s\in S$ and $f\in I$.
	Since $G$ is a left Gr\"obner basis of $I$ there exists $g\in G$ such that $\lm_\preceq(g)\mid\lm_\preceq(f)$.
	In terms of leading exponents, where $(i,(\alpha_1,\alpha_2))=\lexp_\preceq(g)$ and $(j,(\beta_1,\beta_2))=\lexp_\preceq(f)$, this means $i=j$ and $(\alpha_1,\alpha_2)\leq(\beta_1,\beta_2)$, in particular, we have $\alpha_2\leq_2\beta_2$.
	Since $\lexp_{\preceq_2}(\varepsilon(g))=(i,\alpha_2)$ and $\lexp_{\preceq}((s,f))=\lexp_{\preceq_2}(\varepsilon(f))=(j,\beta_2)$ by the previous \Cref{ordering_POT_inheritance_of_leading_exponent_in_second_block}, we have $\lm_{\preceq_2}(\varepsilon(g))\mid\lm_{\preceq_2}(z)$.
\end{proof}

\begin{definition}
	Consider a polynomial $f\in K[x]\setminus K$.
	Since $K[x]$ is a unique factorization domain, $f$ has a representation as a product of finitely many irreducible elements.
	The \emph{square-free part} of $f$, denoted $\sqrt{f}$, is the product of all unique irreducible elements up to associativity that occur in this factorization.
\end{definition}

\begin{remark}
	\Cref{alg:CentralEssentialRationalClosure} is based on its commutative special case which can be found in \cite{BW93}, Table 8.8, as algorithm \textsc{ExtCont}.
\end{remark}

\begin{algorithm}
	\caption{\textsc{CentralEssentialRationalClosure}}
	\label{alg:CentralEssentialRationalClosure}
	\KwIn{A left $A$-submodule $I$ of $A^r$.}
	\KwOut{A left Gr\"obner basis $G\subseteq A^r$ of $I^S$ with respect to $\preceq$.}
	\Begin{
		$H:=\textsc{LeftGr\"obnerBasis}(I,\preceq)$\;
		$h:=\sqrt{\prod_{g\in H}^{}\lc_{\preceq_2}(\varepsilon(g))}\in K[x]\setminus\set{0}$\;
		$k:=\operatorname{Satindex}(I,h)$\;
		$G:=\textsc{LeftGr\"obnerBasis}(I:h^k,\preceq)$\;
		\Return{$G$}\;
	}
\end{algorithm}

In the situation of \Cref{alg:CentralEssentialRationalClosure}, the candidate $h$ is constructed such that for any $g\in H$ there exists $\ell\in\IN$ satisfying $\lc_{\preceq_2}(\varepsilon(g))\mid h^\ell$.

\begin{proposition}
	\Cref{alg:CentralEssentialRationalClosure} terminates and is correct.
\end{proposition}
\begin{proof}
	The saturation index computation is finite since all $G$-algebras are Noetherian, thus termination of the whole algorithm is ensured.
	To prove correctness we have to show that $I^S=I:h^k$.
	
	First, let $f\in I:h^k$, then $h^kf\in I$ and $\varepsilon(f)=(1,f)=(h^k,h^kf)\in S\inv I$.
	Thus we have $f\in\varepsilon\inv(S\inv I)=I^S$, which implies $I:h^k\subseteq I^S$.
	
	For the other inclusion, let $f\in I^S=\varepsilon\inv(S\inv I)$, then $\varepsilon(f)\in S\inv I$.
	Now $\varepsilon(H)$ is a left Gr\"obner basis of $S\inv I$ with respect to $\preceq_2$ by \Cref{ordering_inheritance_of_Gr\"obner_basis_property}.
	Furthermore, \Cref{characterization_of_Gr\"obner_bases_in_terms_of_NF} implies that $\LeftNF(\varepsilon(f)|\varepsilon(H))=0$.
	We now prove $f\in I:h^k$ by an induction on the minimal number $N\in\IN$ of steps necessary in the left normal form algorithm given in \Cref{alg_LeftNF} to reduce $\varepsilon(f)$ to zero:
	\begin{description}[leftmargin=0cm]
		\item[Induction base:]
			If $N=0$, then $\varepsilon(f)=0$.
			Since $\varepsilon$ is injective we have $f=0$, which trivially implies $f\in I:h^\infty$.
		\item[Induction hypothesis:]
			Assume that for any $\tilde{f}\in I^S$, such that $\varepsilon(\tilde{f})$ can be reduced to zero in $N-1$ steps by the left normal form algorithm with respect to $\varepsilon(H)$, we have $\tilde{f}\in I:h^k$.
		\item[Induction step:]
			Let $f\in I^S$ such that the left normal form algorithm needs at least $N$ steps to reduce $\varepsilon(f)$ to zero with respect to $\varepsilon(H)$.
			Then there exists $g\in H$ such that $\lm_{\preceq_2}(\varepsilon(g))\mid\lm_{\preceq_2}(\varepsilon(f))$.
			Let $(i_f,\alpha)=\lexp_{\preceq_2}(\varepsilon(f))$ and $(i_g,\beta)=\lexp_{\preceq_2}(\varepsilon(g))$, then $i_g=i_f$ and
			\[
				t
				:=\varepsilon(f)-\frac{\lc_{\preceq_2}(\varepsilon(f))}{\lc_{\preceq_2}(\varepsilon(y^{\alpha-\beta}g))}\rho(y^{\alpha-\beta})\varepsilon(g)
				\in S\inv A^r
			\]
			can be reduced to zero in $N-1$ steps with respect to $\varepsilon(H)$.
			Since the relations between the variables in $A$ have the form $y_jy_i=c_{ij}y_iy_j+d_{ij}$ for some $c_{ij}\in K\setminus\set{0}$ and $d_{ij}\in A$ such that $\lexp_{\preceq}(d_{ij})\prec\lexp_{\preceq}(y_iy_j)$, we have
			\[
				\lc_{\preceq_2}(\varepsilon(y^{\alpha-\beta}g))
				=u\cdot\lc_{\preceq_2}(\varepsilon(g))
			\]
			for some $u\in K\setminus\set{0}$, which is just the product of all $c_{ij}$ that occur while bringing $\varepsilon(y^{\alpha-\beta}g)$ in standard monomial form, and thus
			\[
				t
				=\varepsilon(f)-\frac{\lc_{\preceq_2}(\varepsilon(f))}{u\lc_{\preceq_2}(\varepsilon(g))}\rho(y^{\alpha-\beta})\varepsilon(g).
			\]
			Since $\lc_{\preceq_2}(\varepsilon(g))$ divides a power of $h$, there exists $\ell\in\IN$ such that
			\[
				c
				:=\frac{h^\ell}{u\lc_{\preceq_2}(\varepsilon(g))}
				\in K[x]\setminus\set{0}
			\]
			and therefore
			\[
				\tilde{f}
				:=h^\ell f-c\lc_{\preceq_2}(\varepsilon(f))y^{\alpha-\beta}g
				\in I^S,
			\]
			since $f\in I^S$ by assumption and $g\in H\subseteq I\subseteq I^S$.
			Now
			\[\begin{split}
				h^\ell t
				&=h^\ell\varepsilon(f)-\frac{h^\ell}{u\lc_{\preceq_2}(\varepsilon(g))}\lc_{\preceq_2}(\varepsilon(f))\rho(y^{\alpha-\beta})\varepsilon(g)\\
				&=h^\ell\varepsilon(f)-c\lc_{\preceq_2}(\varepsilon(f))\rho(y^{\alpha-\beta})\varepsilon(g)\\
				&=\varepsilon(h^\ell f-c\lc_{\preceq_2}(\varepsilon(f))y^{\alpha-\beta}g)\\
				&=\varepsilon(\tilde{f}),
			\end{split}\]
			thus we can apply the induction hypothesis: we have $\tilde{f}\in I^S$ such that $\varepsilon(\tilde{f})=h^\ell t$ can be reduced to zero in $N-1$ steps with respect to $\varepsilon(H)$, since $h^\ell\in S$ is invertible in $S\inv A$ and thus does not change the reducibility of $t$.
			This gives us $\tilde{f}\in I:h^k$ or $h^k\tilde{f}\in I$.
			Now
			\[
				h^{\ell+k}f
				=h^k\tilde{f}+h^kc\lc_{\preceq_2}(\varepsilon(f))y^{\alpha-\beta}g
				\in I
			\]
			implies $f\in I:h^{\ell+k}=I:h^k$, which shows $I^S\subseteq I:h^k$.\qedhere
	\end{description}
\end{proof}

\begin{remark}\label{rem:Alg10isMoreGeneral}
	Let $R$ be a commutative principal ideal domain, which is a computable ring with the field of fractions $Q$.
	By replacing $K[x]$ with $R$ in \Cref{alg:CentralEssentialRationalClosure}, we observe that the same proof can be applied to the following more general situation: suppose that $A$ is a $G$-algebra over $Q$ which contains $Q[x]$.
	Assume further that $c_{ij}$ and all the coefficients of $d_{ij}$ are in $R$, then we define $A_R$ to be an $R$-algebra subject to the same relations as $A$.

	Consider the algorithm applied for a $G$-algebra $A$ over $Q$, a left submodule $I\subseteq A^r$, an algebra $A_R$ over $R$, and $S = R\setminus\set{0}$.
	We replace $K[x]$ with $R$ and do not need to employ an antiblock ordering.
	After computing a left Gr\"obner basis $H$ of $I$ over the field $Q$ we can assume that no denominators are present in $H$.
	Now the candidate $h\in R\setminus\set{0}$ and the rest of the algorithm is the same.
	Also the proof carries almost verbatim with only one modification: since $c:=\frac{h^l}{u\lc_{\preceq_2}(\varepsilon(g))}\in Q\setminus\set{0}$ is a fraction, while $h,u,\lc_{\preceq_2}(\varepsilon(g))\in R$, we just have to replace $\tilde{f}$ with $\hat{f}:=u\cdot\tilde{f}\in I^S$.
	Of course, the computations of the saturation index and the final left Gr\"obner basis happen over $R$ (which would require a special implementation in comparison to the case of ground fields).
	A very natural application of the described algorithm is for $R=\mathbb{Z}$.
\end{remark}

We implemented \Cref{alg:CentralEssentialRationalClosure} using the computer algebra system \textsc{Singular:Plural} (\cite{Plural}) and used it on problems coming from e.g. $D$-module theory:

\begin{example}
	In $\mathcal{D}_3[s]$, the third Weyl algebra over the field $K=\mathbb{Q}$ with an additional commutative variable $s$, we compute the $K[s]\setminus\set{0}$-closure of the left ideal $L_{1}$ which is generated by the elements of order 1 in the derivatives
	\[
		x\partial_x+y\partial_y-5s,xz\partial_z+y\partial_z-xs,y^2z^2\partial_z+y^3\partial_x+x^3\partial_y-y^2zs-x^2\partial_z.
	\]
	The candidate used for saturation is $25s^2+25s+6=(5s+2)(5s+3)$ and the saturation is reached after one step taking barely any time.  The resulting ideal $L$ is a part of the annihilating ideal $I$ of the special function $((xz+y)(x^4-y^4))^s$. Notably, the factor $5s+2$ is still present among the leading coefficients of generators of $L$. Moreover, $L_{1} \subsetneq L$ shows that $I$ cannot be generated by the elements of order 1 only.
\end{example}

The following establishes a sufficient condition for computing local closures iteratively:

\begin{lemma}
	Let $S_1$ and $S_2$ be left denominator sets in a ring $R$, $M$ a left $R$-module and $P$ a left $R$-submodule of $M$.
	If $S_1S_2=S_2S_1$, then $P^{S_1S_2}=(P^{S_1})^{S_2}$.
\end{lemma}
\begin{proof}
	Since $S_1$ and $S_2$ are subsets of $S_1S_2$, $(P^{S_1})^{S_2}\subseteq P^{S_1S_2}$ immediately follows.
	For the other inclusion, if $m\in P^{S_1S_2}$, then there exist $s_1\in S_1$ and $s_2\in S_2$ such that $s_1s_2m\in P$, thus $m\in(P^{S_1})^{S_2}$.
\end{proof}

\begin{remark}\label{remYZ}
	Let $R$ be a commutative principal ideal domain and $R[x]$ a polynomial ring with field of fractions $Q(x)$.
	Consider a left ideal $L$ in the single Ore extension $Q(x)[\partial;\sigma,\delta]$, then in \cite{YZ16} one finds an algorithm for computing the contraction $Q(x)[\partial; \sigma, \delta] L \cap R[x][\partial; \sigma, \delta]$.
	We recognize the latter as the $R[x]\setminus\{0\}$-closure of $L$. 
	In the general setting, addressed in \Cref{rem:Alg10isMoreGeneral}, we can compute the $R[x]\setminus\set{0}$-closure of a submodule of $A^r$ in two steps: let $S_1=Q[x]\setminus\set{0}$ and $S_2 = R\setminus\set{0}$.
	Then $I^{R[x]\setminus\set{0}} = (I^{S_1})^{S_2}$ holds by the following result.
	The left submodule $I^{S_1}$ can be computed with the modified \Cref{alg:CentralEssentialRationalClosure} as explained in \Cref{rem:Alg10isMoreGeneral}.
\end{remark}

\subsection{Central Weyl closure}
\label{subsectCentralWeylclosure}

Let $K$ be a field and $A_n^K$ the $n$-th Weyl algebra over $K$ in the variables $x=\set{x_1,\ldots,x_n}$ and $\partial=\set{\partial_1,\ldots,\partial_n}$.
In this section, we utilize the central closure algorithm presented above in combination with the famous Weyl closure algorithm from \cite{tsai_thesis} to give an algorithm to compute the $K[x,s]\setminus\set{0}$-closure of a left ideal in the algebra $A_n^K[s]:=A_n^K\tensor_KK[s]$, where $s=\set{s_1,\ldots,s_m}$ is a set of additional commutative indeterminates and $S:=K[s]\setminus\set{0}$.
Lastly, let $\rho:=\rho_{S,A_n^K[s]}$.

\begin{remark}
	Note that the extended ring $A_n^K[s]:=A_n^K\tensor_KK[s]$ is no longer a Weyl algebra, but still a $G$-algebra.
	Localizing $A_n^K[s]$ at $S$ yields $A_n^K(s)$, which is isomorphic to $A_n^{K(s)}$, the $n$-th Weyl algebra over the field $K(s)$.
	The localization map $\rho$ is injective, since $A_n^K[s]$ is a domain.
	
	Moreover, most of computations are done over $K[s]$, though mathematically we work over $K(s)$; retaining more generators with coefficients in $K[s]$ is a classical strategy, while working
	with localizations.
\end{remark}

\begin{definition}
	Let $I$ be a left ideal in $A_n^K[s]$.
	Define
	\[
		G
		:=I^{K[x,s]\setminus\set{0}}
		=\LSat_{K[x,s]\setminus\set{0}}^{A_n^K[s]}(I)
		=\set{r\in A_n^K[s]\mid\exists w\in K[x,s]\setminus\set{0}:wr\in I}
	\]
	and
	\[
		H
		:=\LSat_{K(s)[x]\setminus\set{0}}^{A_n^{K(s)}}((I^S)^e)
		=\set{r\in A_n^{K(s)}\mid\exists w\in K(s)[x]\setminus\set{0}:wr\in(I^S)^e}.\qedhere
	\]
\end{definition}

\begin{lemma}\label{central_Weyl_closure_lemma}
	We have $H=G^e$ with respect to $\rho$.
\end{lemma}
\begin{proof}
	Let $r\in G$, then there exists $w\in K[x,s]\setminus\set{0}$ such that $wr\in I$.
	Now
	\[
		\rho(w)\rho(r)
		=\rho(wr)
		\in\rho(I)
		\subseteq\rho(I^S)
		\subseteq(I^S)^e
	\]
	and $\rho(w)\in\rho(K[x,s]\setminus\set{0})\subseteq K(s)[x]\setminus\set{0}$, thus $\rho(r)\in H$ and therefore $\rho(G)\subseteq H$.
	Since $H$ is a left ideal in $A_n^{K(s)}$ it contains $G^e$, the left ideal generated by $\rho(G)$, thus $G^e\subseteq H$.\\
	Now let $(t,r)\in H$, where $t\in S$ and $r\in A_n^K[s]$, then there exists $(q,w)\in K(s)[x]\setminus\set{0}$, where $q\in S$ and $w\in K[x,s]\setminus\set{0}$, such that $(q,w)\cdot(t,r)\in(I^S)^e$.
	Since $w$ and $t$ commute we have $(q,w)\cdot(t,r)=(tq,wr)$.
	Now $(tq,wr)\in(I^S)^e$, therefore there exist $\tilde{t}\in S$ and $p\in I^S$ such that $(tq,wr)=(\tilde{t},p)$.
	This implies the existence of $\bar{t}\in S$ and $\bar{r}\in A_n^K[s]$ such that $\bar{t}tq=\bar{r}\tilde{t}$ and $\bar{t}wr=\bar{r}p\in I^S$, thus there exist $\hat{t}\in S$ satisfying $\hat{t}\bar{t}wr\in I$.
	Since $\hat{t}\bar{t}w\in K[x,s]\setminus\set{0}$ we have $r\in G$, thus $(t,r)=(t,1)\cdot(1,r)=(t,1)\cdot\rho(r)\in G^e$ and therefore $H\subseteq G^e$.
\end{proof}

\begin{algorithm}[H]
	\caption{\textsc{CentralWeylClosure}}
	\label{alg_CentralWeylClosure}
	\KwIn{A left ideal $I$ of $A_n^K[s]$.}
	\KwOut{The $K[x,s]\setminus\set{0}$-closure of $I$.}
	\Begin{
		compute $I^S$ via central closure in $A_n^K[s]$\tcp*{central closure}
		extend $I^S$ from $A_n^K[s]$ to $A_n^{K(s)}$\tcp*{extension}
		compute the $K(s)[x]\setminus\set{0}$-closure $H=\ideal{\rho(h_1),\ldots,\rho(h_k)}$ of $(I^S)^e$ in $A_n^{K(s)}$ via Weyl closure\tcp*{Weyl closure}
		let $F:=\ideal{h_1,\ldots,h_k}\subseteq A_n^K[s]$\tcp*{primitive contraction}
		compute $F^S$ via central closure in $A_n^K[s]$\tcp*{central closure}
		\Return{$F^S$}\;
	}
\end{algorithm}

\begin{figure}[h]
	\[\begin{tikzcd}[column sep=7em,row sep=3em]
		\text{in }A_n^K[s]:
		&[-5em]I
			\arrow[]{r}{}[swap]{\text{central closure}}
			\arrow[bend left=20,dashed]{rrr}{}
		&I^S
			\arrow[]{d}{\text{extension}}
		&F
			\arrow[]{r}{}[swap]{\text{central closure}}
		&G=F^S
		\\
		\text{in }A_n^{K(s)}:
		&&(I^S)^e
			\arrow[]{r}{}[swap]{\text{Weyl closure}}
		&H
			\arrow[]{u}{}[swap]{\text{primitive contraction}}
	\end{tikzcd}\]
	\caption{Idea of \Cref{alg_CentralWeylClosure}.}
\end{figure}

\begin{proposition}
	\Cref{alg_CentralWeylClosure} terminates and is correct.
\end{proposition}
\begin{proof}
	Termination is obvious.
	Since $A_n^{K(s)}$ is Noetherian, the ideal $H$ is finitely generated: let $H=\ideal{(a_1,h_1),\ldots,(a_k,h_k)}$ for some $a_i\in S$ and $h_i\in A_n^K[s]$.
	Since $(a_i,1)$ is a unit in $A_n^{K(s)}$ for all $i$, we have $H=\ideal{(1,h_1),\ldots,(1,h_k)}=\ideal{\rho(h_1),\ldots,\rho(h_k)}$.
	It remains to show that $G=F^S$.\\
	First, let $g\in G$, then $\rho(g)\in G^e=H$ by \Cref{central_Weyl_closure_lemma}, thus there exist $t_i\in S$ and $r_i\in A_n^K[s]$ such that
	\[
		\rho(g)
		=\sum_{i=1}^{k}(t_i,r_i)\rho(h_i)
		=\sum_{i=1}^{k}(t_i,r_ih_i)
		=(t,\sum_{i=1}^{k}\tilde{r}_ir_ih_i)
	\]
	for some common left denominator $t\in S$ and $\tilde{r}_i\in A_n^K[s]$.
	Now
	\[
		\rho(tg)
		=\rho(t)\rho(g)
		=\rho(t)(t,\sum_{i=1}^{k}\tilde{r}_ir_ih_i)
		=\rho\left(\sum_{i=1}^{k}\tilde{r}_ir_ih_i\right).
	\]
	Since $\rho$ is injective we have $tg=\sum_{i=1}^{k}\tilde{r}_ir_ih_i\in F$, therefore $g\in F^S$ and thus $G\subseteq F^S$.\\
	For the second inclusion let $y\in F^S$, then there exists $t\in S$ such that $ty\in F$.
	Now $\rho(ty)\in F^e\subseteq H$, since $F\subseteq H^c$ implies $F^e\subseteq(H^c)^e=H$.
	The ideal $H$ is left $K(s)[x]\setminus\set{0}$-saturated by construction, thus $\rho(t)\in K(s)[x]\setminus\set{0}$ implies $\rho(y)\in G^e$.
	Then $y\in(G^e)^c=G^S=G$ since $G$ is left $S$-saturated, therefore $F^S\subseteq G$.
\end{proof}

\begin{remark}
\label{remWeylClosureAlgorithms}
Let $I\subseteq A_n^{K}=:D$ be a left ideal.
In \cite{tsai_thesis}, two Weyl closure algorithms have been presented.
The optimized one only works for ideals of finite \emph{holonomic rank} (i.e. those ideals $J$, such
that $\dim_{K(x,s)} S^{-1}D/S^{-1}D J < \infty$ for $S=K[x,s]\setminus\{0\}$) and is based on a $D$-module-theoretic (monoidal) localization algorithm.
It has been implemented in \textsc{Macaulay2} and in \textsc{Singular:Plural}.

The general one, which works for any ideal $I$, is much harder, since it relies on a complicated algorithm to determine monoidal $[f]$-torsion of certain finitely generated $D$-modules associated with $I$.
Notably, the general algorithm has not been implemented in any computer algebra system.
\end{remark}

\subsection{Application: computing \texorpdfstring{$\Ann_{D[s]} f^s$}{AnnD[s]fs} via \textsc{CentralWeylClosure}}
\label{AnnfsViaCentralWeylclosure}

The algorithm above has a nice application: namely, the computation of the annihilator of the special function $f^s:=f_1^{s_1}\cdot\ldots\cdot f_r^{s_r}$ in the algebra $A_n^{K}[s]=D[s]$, where $f_i \in K[x]$.
It is a left ideal, denoted by $\Ann_{D[s]} f^s$.
For simplicity of the presentation, let $r=1$, i.~e. $f=f_1\in K[x]$ and $s_1=s$.
We refer the reader to \cite{ABLMS} for details.
Analytic considerations deliver the following additional information:

\begin{proposition} 
	With notations as above, let $S=K[x,s]\setminus\set{0}\subseteq D[s]$ and $S'=K[x]\setminus\set{0}$.
	\begin{enumerate}[(a)]
	\item
		The left ideal $\Ann_{D[s]} f^s$ is $S$-closed.
	\item
		Both $G_1:=\set{f{\partial_i}-\frac{\partial f}{\partial x_i}s\mid 1\leq i\leq n} $ and $G_2$ (see below) generate localized ideals $S^{-1} \Ann_{D[s]} f^s$ and $S'^{-1} \Ann_{D[s]} f^s$, respectively.
	\end{enumerate}
	Consider the $K[x,s]$-module of syzygies of the tuple $(f, s\frac{\partial f}{\partial x_1},\ldots,s\frac{\partial f}{\partial x_n})$.
	Since it is finitely generated, let $\set{g_1,\ldots,g_t}\subseteq K[x,s]^{n+1}$ be its generating set.
	Then
	\[
		G_2
		=\set{a_0+a_1{\partial_1}+\cdots+a_n{\partial_n}\mid(a_0,a_1,\ldots,a_n)=g_i,1\leq i\leq n}.
	\]
\end{proposition}

Note that $G_2$ is generated by vector fields of order one in $D[s]$.
Also, $\Ann_{D[s]} f^s$ is of holonomic rank 1.
Therefore we can at long last provide the following alternative to the algorithm of Brian\c{c}on-Maisonobe (\cite{BM02}):

\begin{algorithm}[H]
	\caption{\textsc{AnnFsViaWeylClosure}}
	\label{alg_AnnFsViaWeylClosure}
	\KwIn{A set $G\subseteq D[s]$ and $f^s$ as above.}
	\KwOut{A generating set of $\Ann_{D[s]}f^s\subseteq D[s]$.}
	\Begin{
		ensure that $G$ generates $(K[x,s]\setminus\set{0})\inv\Ann_{D[s]} f^s$\;
		compute $F:=\textsc{CentralWeylClosure}(G)$\;
		\Return{$F$}\;
	}
\end{algorithm}

\subsection{Central geometric closure}

Consider the setting from \Cref{CERC_setting}, but we are now interested in computing the closure $I^T$, where $T:=K[x]\setminus\mf{p}$ for some prime ideal $\mf{p}$ in $K[x]$.
By construction we have $I\subseteq I^T\subseteq I^S$ and we can characterize when the second inclusion is in fact an equality:

\begin{lemma}
	We have $I^T=I^S$ if and only if $\Ann_T(I^S/I)\neq\emptyset$.
\end{lemma}
\begin{proof}
	Recall that for an $A$-module $I$ and a subset $P$ of $A$, $\Ann_P(I):=\set{p\in P\mid pI=0}$.
	Let $W\in\set{S,T}$, then $I^W$ is finitely generated by some elements $f_1,\ldots,f_k\in A$ and we have that $\Ann_W(I^W/I)$ is non-trivial: since $f_i\in I^W$ there exist $w_i\in W$ such that $w_if_i\in I$, so $w_1\cdot\ldots\cdot w_k\in\Ann_W(I^w/I)$ due to $W$ being central in $A$.
	If $I^S=I^T$ then $\Ann_T(I^S/I)=\Ann_T(I^T/I)\neq\emptyset$.
	On the other hand, let $t\in\Ann_T(I^S/I)$ and $r\in I^S$, then $tr\in I$ and thus $r\in I^T$, which shows $I^S=I^T$.
\end{proof}

Note that $\Ann_T(I^S/I)\neq\emptyset$ is equivalent to $\Ann_B(I^S/I)\nsubseteq\mf{p}$ (recall, that $B$ is identified with $K[x]$ in our setup from \Cref{CERC_setting}) and the latter can be checked algorithmically, since $I^S$ is computable via \Cref{alg:CentralEssentialRationalClosure}.

Nevertheless there are situations where neither inclusion is strict:

\begin{example}
	Consider $I=\leftideal{A}{x(x-1)\partial}$, where $A$ is the first Weyl algebra in $x$ and $\partial$.
	Then $I^S=\leftideal{A}{\partial}$ and $I^T=\leftideal{A}{x\partial}$, if we choose $\mf{p}=\leftideal{K[x]}{x}$, which leads to $I\subsetneq I^T\subsetneq I^S$.
\end{example}

Further advances towards an algorithm for computing $I^T$ are the subject of ongoing research.

\section{Conclusion}

We have provided several algorithms for solving the intersection problem and for computing local closure in various settings with respect to Ore sets with enough commutativity. 
In particular, it follows that arithmetic within the localization of a commutative polynomial algebra is constructive and can be used also in homomorphic images of such algebras inside non-commutative algebras.

At the end of our paper \cite{HL_ISSAC18}, we posed the following questions: does there exist an algorithm to compute\ldots

\begin{itemize}
	\item
		the closure in the case of geometric localization without invoking primary decomposition?
	\item
		the central geometric closure?
	\item
		the geometric closure in the Weyl algebra tensored with a commutative polynomial ring?
\end{itemize}

As it turns out, the first question has been answered since then in \cite{IK18}, which was published just a few months later than \cite{HL_ISSAC18}.
The main tool used is the \emph{double ideal quotient} $I : (I : P)$, where $P$ is a prime ideal, and its variants.
This opens a perspective towards better versions of algorithms, which rely on primary decomposition; among other, for the computation of the symbolic power of an ideal.

However, the other questions we posed are still open.

\section{Acknowledgements}

The authors are grateful to Thomas Kahle (Magdeburg), Gerhard Pfister (Kaiserslautern), Anne Fr\"uhbis-Kr\"uger (Hannover and Oldenburg), Jorge Mart\'in-Morales (Zaragoza) and Simone Bamberger (Aachen) for fruitful discussions.

The authors have been supported by Project I.12 and Project II.6 respectively of SFB-TRR 195 ``Symbolic Tools in Mathematics and their Applications'' of the German Research Foundation (DFG).

\bibliographystyle{elsarticle-harv}
\bibliography{bibliography}

\end{document}